\documentclass[11pt]{amsart}

\usepackage{amsfonts}
\usepackage{amssymb}
\usepackage{color}
\usepackage{amsmath, amsthm, latexsym}
\usepackage[all]{xy}
\usepackage{hyperref}
\usepackage{graphicx}
\usepackage[margin=1.5in]{geometry}

\def \C{\mathbb{C}}
\def \Z{\mathbb{Z}}
\def \R{\mathbb{R}}

\def \Q{\mathbb{Q}}

\def \D{\mathcal{D}}

\def \vol{\textup{vol}}
\def \trop{\textup{trop}}

\def \SL{\textup{SL}}
\def \GL{\textup{GL}}

\def \codim{\textup{codim}}
\def \pr{\textup{pr}}

\theoremstyle{plain}
\newtheorem{Th}{Theorem}[section]

\theoremstyle{definition}
\newtheorem{Ex}[Th]{Example}
\newtheorem{Def}[Th]{Definition}
\newtheorem{Rem}[Th]{Remark}

\newtheorem{Prob}[Th]{Problem}

\begin{document}
\title{A short survey on Newton polytopes, tropical geometry and ring of conditions of algebraic torus}

\author{Kiumars Kaveh}
\address{Department of Mathematics, University of Pittsburgh,
Pittsburgh, PA, USA.}
\email{kaveh@pitt.edu}
\author{A. G. Khovanskii}
\address{Department of Mathematics, University of Toronto, Toronto,
Canada; Moscow Independent University, Moscow, Russia.}
\email{askold@math.utoronto.ca}
\maketitle

\section{Introduction}
The purpose of this note is to give an exposition of some interesting combinatorics and convex geometry concepts that appear in algebraic geometry in relation to counting the number of solutions of a system of polynomial equations in several variables over complex numbers. This approach belongs to relatively new, and closely related branches of algebraic geometry which are usually referred to as {\it topical geometry} and {\it toric geometry}. These areas make connections between the study of algebra and geometry of polynomials and the combinatorial and convex geometric study of piecewise linear functions. 

Throughout, the coefficients of the polynomials we consider are in the field of complex numbers $\C$. We denote by $\C^* = \C \setminus \{0\}$ the multiplicative group of nonzero complex numbers. The $n$-fold product $(\C^*)^n$ is a multiplicative group 
(more precisely, it is an affine algebraic group). It is called an {\it algebraic torus}. The usual topological torus is $(S^1)^n = S^1 \times \cdots \times S^1$ where $S^1$ is the unit circle (the familiar $2$-dimensional torus or donut shape is $S^1 \times S^1$). The topological torus $(S^1)^n$ sits inside $(\C^*)^n$ and it is regarded as the ``complexification'' of $(S^1)^n$.

This note is written with the hope of being accessible to undergraduate as well as advance high school students in mathematics.

\section{Newton polytope}
To a Laurent polynomial $f(x_1, \ldots, x_n) \in \C[x_1^{\pm 1}, \ldots, x_n^{\pm}]$ one can associate a polytope in $\R^n$ called the {\it Newton polytope of $f$}. 
(Recall that a Laurent polynomial is a linear combination of monomials in $x_1, \ldots, x_n$ where the exponents are allowed to be negative integers as well.)

\begin{Def}[Newton polytope]   \label{def-Newton-polytope}
Let $f(x_1, \ldots, x_n) = \sum_{\alpha = (a_1, \ldots, a_n) \in \Z^n} c_\alpha x_1^{a_1} \cdots x_n^{a_n}$. The {\it Newton polytope} $\Delta(f)$ of $f$ is the polytope in $\R^n$ obtained by taking the convex hull of all the $\alpha \in \Z^n$ with $c_\alpha \neq 0$.   
\end{Def}  

The Newton polytope $\Delta(f)$ is clearly a lattice polytope, that is, all its vertices belong to $\Z^n$.

\begin{Ex}
Let $f(x, y) = y^2 + a_0 + a_2x^2 + a_3x^3$, where $a_0, a_2, a_3 \neq 0$. Then the Newton polytope of $f$ is the polygon in Figure \ref{fig-Newton-polygon}.
\begin{figure}[ht]  \label{fig-Newton-polygon}
\includegraphics[width=6cm]{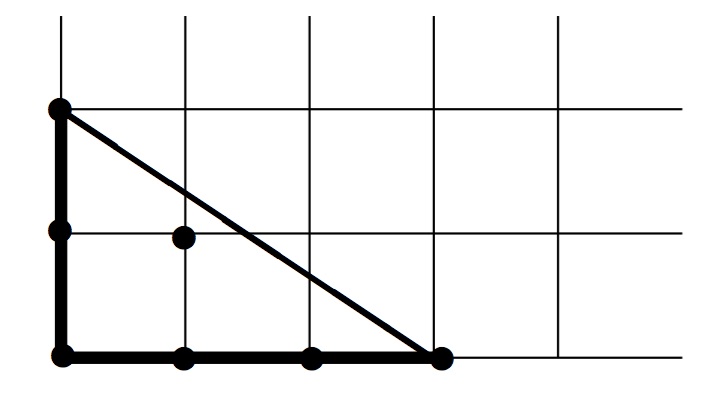} 
\caption{Newton polygon of $f(x, y) = y^2 + a_0 + a_2x^2 + a_3x^3$} 
\end{figure} 
\end{Ex}

The idea of Newton polytope appears in the work of Sir Isaac Newton. The above definition was introduced by the Moscow school of Newton polyhedra theory during the 1970s.

\section{Minkowski sum and Minkowski mixed volume of polytopes}   \label{sec-mixed-vol}
There are natural operations of multiplication by a positive scalar and addition on the set $\mathcal{P} = \mathcal{P}_n$ of all convex polytopes in $\R^n$. Let $c > 0$ and let $\Delta \in \mathcal{P}$ be a convex polytope. Then:
$$c \Delta = \{cx \mid x \in \Delta\},$$ is again a convex polytope. Similarly, let $\Delta_1, \Delta_2 \in \mathcal{P}$ be two convex polytopes. The {\it Minkowski sum} $\Delta_1 + \Delta_2$ of these two polytopes is defined by:
$$\Delta_1 + \Delta_2 = \{ x_1 + x_2 \mid x_1 \in \Delta_1,~ x_2 \in \Delta_2\}.$$

\begin{Prob}
Show that $\Delta_1 + \Delta_2$ is in fact a convex polytope. Moreover, if $\Delta_1$ and $\Delta_2$ are lattice polytopes (that is, their vertices belong to 
$\Z^n$) then $\Delta_1 + \Delta_2$ is also a lattice polytope.
\end{Prob}

\begin{figure}[ht]
\includegraphics[width=12cm]{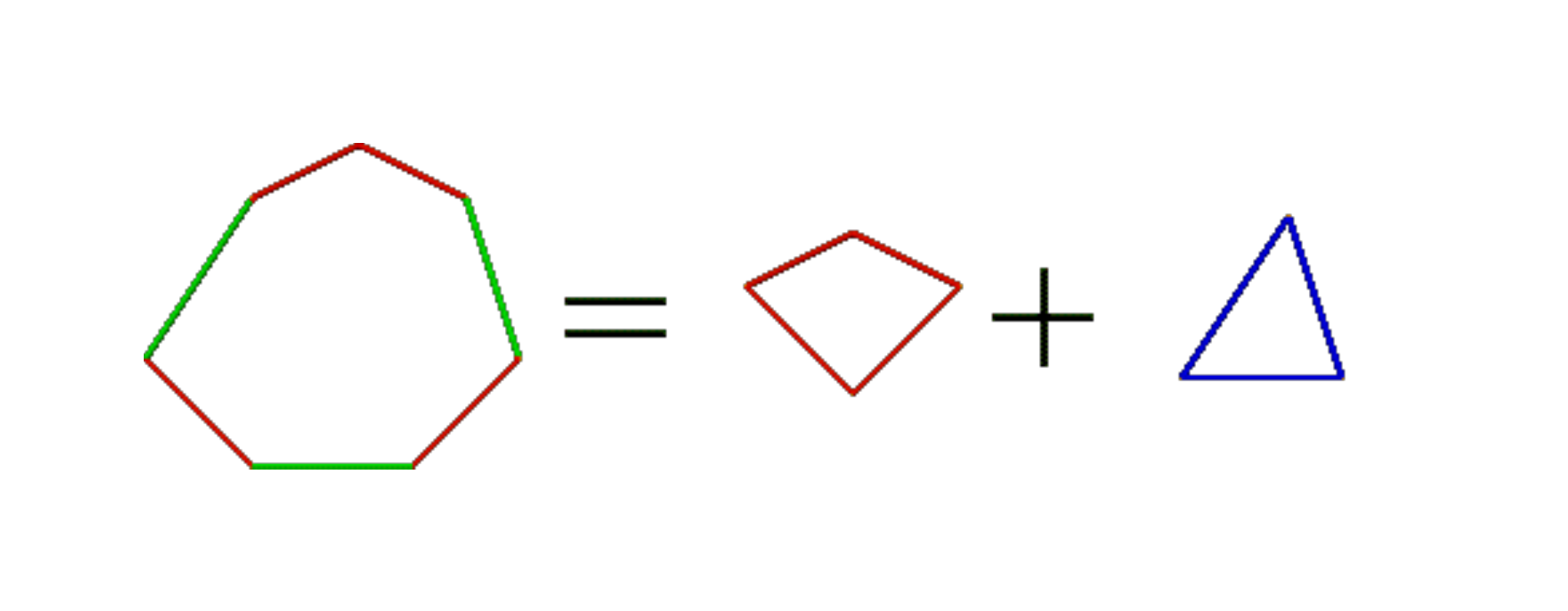}
\caption{Minkowksi sum of a triangle and a quadrangle}
\end{figure}

{We point out that if $\Delta$ is a polytope, the set $-\Delta = \{-x \mid x \in \Delta \}$ is not the inverse of $\Delta$ under Minkowski sum. That is, $\Delta + (-\Delta)$ almost never is equal to $\{0\}$. That is why scalar multiplication of polytopes with negative numbers is not so interesting for us.}

A set which has an addition operation as well as a multiplication with positive scalars is usually called a {\it cone}. The above discussion shows that the space $\mathcal{P}$, of all convex polytopes in $\R^n$, is a cone (with Minkowski sum and multiplication with positive scalars). One can show that the function $\Delta \mapsto \vol_n(\Delta)$ defines a homogeneous polynomial function of degree $n$ on the cone $\mathcal{P}$. This then implies that there is a function $V: \underbrace{\mathcal{P} \times \cdots \times \mathcal{P}}_{n} \to \R$ which satisfies the following properties:
\begin{itemize}
\item[(1)] $V(\Delta_1, \ldots, \Delta_n)$ is a symmetric function in the $\Delta_i$.
\item[(2)] For any polytope $\Delta \in \mathcal{P}$, $V(\Delta, \ldots, \Delta) = \vol_n(\Delta)$.
\item[(3)] $V$ is multi-linear with respect to Minkowski sum and multiplication with positive scalars. That is, for any $c > 0$ and polytopes $\Delta_1', \Delta_1'', \Delta_2, \ldots, \Delta_n \in \mathcal{P}$ we have:
$$V(c\Delta_1' + \Delta''_1, \Delta_2, \ldots, \Delta_n) = cV(\Delta_1', \Delta_2, \ldots, \Delta_n) + V(\Delta''_1, \Delta_2, \ldots, \Delta_n).$$
\end{itemize}
The above properties uniquely determine the function $V$. It is called the {\it Minkowski mixed volume} of convex polytopes.

In fact, if $f: V \to \R$ is any homogeneous polynomial of degree $n$ on a vector space $V$, there is a unique multilinear function $F: \underbrace{V \times \cdots \times V}_n \to \R$ such that for any $v \in V$ we have $F(v, \ldots, v) = f(v)$. The multilinear function $F$ is usually called the {\it polarization of $f$}. It can be computed as:
\begin{equation*}
(-1)^n n! F(v_1, \ldots, v_n) = - \sum_{i} f(v_i) + \sum_{i < j} f(v_i + v_j) + \cdots 
+ (-1)^n f(v_1 + \cdots + v_n). 
\end{equation*}

In particular, one has the following formula which computes the mixed volume in terms of usual volumes. Let $\Delta_1, \ldots, \Delta_n \in \mathcal{P}$ be $n$ convex polytopes. We have:
\begin{multline*}
(-1)^n n! V(\Delta_1, \ldots, \Delta_n) = - \sum_{i} \vol_n(\Delta_i) + \sum_{i < j} \vol_n(\Delta_i + \Delta_j) + \cdots +  \\
+ (-1)^n \vol_n(\Delta_1 + \cdots + \Delta_n). 
\end{multline*}

\section{Number of solutions of a system of equations and the BKK theorem}   \label{sec-BKK}
There is a beautiful theorem due to Bernstein and Kushnirenko that gives an answer for the number of solutions of a system of Laurent polynomial equations in terms  of mixed volume of convex polytopes. This theorem is also sometime called the BKK theorem or the Bernstein-Kushnirenko-Khovanskii theorem (see\cite{Kushnirenko}, \cite{Bernstein} and \cite[Section 5]{Cox}). 

Let $A_1,...,A_n$ be finite subsets in $\Z^n$. For each $i$, let $L_{A_i}$ be subspace of Laurent polynomials given by the span of the monomials $x_1^{a_1} \cdots x_n^{a_n}$ for all $\alpha=(a_1, \ldots, a_n) \in A_i$. Also let $\Delta_i = \Delta(A_i)$ be the convex hull of $A_i$.

\begin{Th}[BKK theorem]   \label{th-BKK}
The number of solutions $x \in (\C^*)^n$ of a system of equations $f_1(x) = \cdots = f_n(x) = 0$ where $f_i$ is a generic element of $L_{A_i}$, is equal to $n!$ times the mixed volume $V(\Delta_1, \ldots, \Delta_n)$.
\end{Th}

\section{Ring of conditions in the torus $(\C^*)^n$}
The classical intersection theory of varieties associates to a variety $M$ its so-called {\it Chow groups}. The $k$-th Chow group $A^k(M)$ consists of formal linear combinations of codimension $k$ subvarieties in $M$ up to a certain equivalence called {\it rational equivalence}. When the variety $M$ is smooth there is a notion of multiplication among elements of Chow groups and one can form the {\it Chow ring} $A^*(M)$. This multiplication is defined by (properly) intersecting subvarieties in $M$ and recording the so-called {\it intersection multiplicities}.\footnote{The Chow ring is an algebraic version of the homology (intersection) ring on a smooth compact topological manifold (which is dual to its cohomology ring).}

This intersection theory works best when dealing with compact varieties (or in algebraic geometric terminology, complete varieties). We would like here to consider a variant of intersection theory for the (non-compact) variety $(\C^*)^n$. More generally this variant of intersection theory works for other groups such as $\GL(n, \C)$, the group of invertible $n \times n$ complex matrices. In this intersection theory one associates a ring $\mathcal{R}((\C^*)^n)$ to the algebraic group $(\C^*)^n$ called the {\it ring of conditions of $(\C^*)^n$}. Similar to the Chow group, the elements of  
$\mathcal{R}((\C^*)^n)$ are formal linear combinations of subvarieties in $(\C^*)^n$ but considered up to a different and stronger equivalence. In the definition of this equivalence one uses the group structure in $(\C^*)^n$. As in the case of Chow rings, one uses the intersection of subvarieties in $(\C^*)^n$ to define the multiplication in the ring of condition. The definition of ring of conditions goes back to DeConcini and Procesi in their fundamental paper \cite{DeConcini-Procesi}. They introduced it as a natural ring in which one can study many classical problems from enumerative geometry (this is related to Hilbert's fifteenth problem).\footnote{They also showed that, for a so-called reductive group $G$, the ring $\mathcal{R}(G)$ can be realized as a limit of Chow rings of all ``good'' compactifications of the group $G$.} 

We will describe the ring of conditions of $(\C^*)^n$. It can be considered as a generalization of the BKK theorem (Theorem \ref{th-BKK}) to subvarieties in $(\C^*)^n$ which are not necessarily hypersurfaces.\footnote{The existence of such a generalization is not unexpected, the point is that the Chow ring of any smooth projective toric variety is generated by its first Chow group which consists of linear combinations of hypersurfaces (in other words the cohomology ring of such a variety is generated in degree 2).}
We will discuss two descriptions of the ring of conditions $\mathcal{R}((\C^*)^n)$ in terms of combinatorial and convex geometric data. One uses the so-called tropical fans and the other uses volume of convex lattice polytopes. 


Let us explain the definition of the ring of conditions.
Recall that a {\it subvariety} $X \subset (\C^*)^n$ is a solution of a finite number of Laurent polynomial equations. That is, we can find Laurent polynomials $f_1, \ldots, f_s \in \C[x_1^{\pm 1}, \ldots, x_n^{\pm 1}]$ such that $X = \{(x_1, \ldots, x_n) \in (\C^*)^n \mid f_i(x_1, \ldots, x_n) = 0,~ i=1, \ldots, s\}$. A subvariety $X$ is {\it irreducible} if it cannot be written as a union of two other subvarieties in a nontrivial way. We first note that if $X \subset (\C^*)^n$ is a subvariety we can move $X$ around by multiplying by any element of $(\C^*)^n$ (it is a multiplicative group):
\begin{Prob}
Let $g \in (\C^*)^n$ and let $X \subset (\C^*)^n$ be any subvariety. Show that $g \cdot X = \{ g \cdot x \mid x \in X\}$ is also a subvariety. 
\end{Prob}

Consider the set $\mathcal{C}$ of {\it algebraic cycles} in $(\C^*)^n$. That is, every element of $\mathcal{C}$ is a formal linear combination $V = \sum_i a_i V_i$ where $a_i \in \Z$ and $V_i$ is an irreducible subvariety. Clearly, with the formal addition operation of cycles, $\mathcal{C}$ is an abelian group. If all the subvarieties  $V_i$ in $V$ have the same dimension $k$ we say that $V$ is a $k$-cycle. For $0 \leq k \leq n$, the subgroup of $k$-cycles is denoted by $\mathcal{C}_k$. For a cycle $V = \sum_i a_i V_i$ and $g \in (\C^*)^n$ we define $g \cdot V$ to be $\sum_i a_i(g \cdot V_i)$. A $0$-cycle is just a formal linear combination of points. If $P = \sum_i a_i P_i$ is a $0$-cycle where the $P_i$ are points, we let $|P| = \sum_i a_i$. 

The intersection of two algebraic subvarieties $X$ and $Y$ is a union of finitely many irreducible algebraic varieties $Z$. Let us suppose that $X$ and $Y$ {\it intersect transversely}, i.e. $X \cap Y$ is a union of irreducible components $Z$ such that $\codim(Z) = \codim(X) + \codim(Y)$ and moreover, $X$ and $Y$ intersect {\it transversely} at generic points of intersection. 
We define the {\it intersection product} $X \cdot Y$ to be the cycle: $$X \cdot Y = \sum_Z Z.$$   

Define an equivalence relation on the set of algebraic cycles as follows. Let $V, V' \in \mathcal{C}$ be algebraic cycles of dimension $m$. Let $Z$ be a subvariety of complementary dimension $n - m$. One knows that for generic $g \in (\C^*)^n$, the subvariety $g \cdot Z$ intersects both $V$ and $V'$ transversely.\footnote{This, in its general form, is known as {\it Kleiman's transversality theorem}. It is a version of the famous Thom's transversality theorem.} Then, for generic $g \in (\C^*)^n$, the intersection products $V \cdot (g \cdot Z)$ and $V' \cdot (g \cdot Z)$ are defined and are $0$-cycles. We define an equivalence relation $\sim$ on algebraic cycles by saying that $V \sim V'$ if 
for any $(n-m)$-cycle $Z$ and generic $g \in (\C^*)^n$ we have:
\begin{equation} \label{equ-numerical-equiv}
 |V \cap (g \cdot Z)| = |V' \cap (g \cdot Z)|.
 \end{equation}
That is, $V \sim V'$ if they intersect general translates of any subvariety of complementary dimension at the same number of points. One verifies that if $X_1, X_2, Y_1, Y_2$ are algebraic cycles such that $X_1 \sim X_2$ and $Y_1 \sim Y_2$ then $X_1 \cdot Y_1 \sim X_2 \cdot Y_2$. Thus the intersection product of transverse subvarieties induces an intersection operation on the quotient $\mathcal{C} / \sim$. The {\it ring of conditions of $(\C^*)^n$} is 
$\mathcal{C} / \sim$ with the ring structure coming from addition and intersection product of cycles.

{More generally, the above definition works if we replace $(\C^*)^n$ with other groups such as $\GL(n, \C)$ or $\SL(n, \C)$ (and in fact for any so-called {\it connected reductive algebraic group $G$}, in this case one considers the left-right action of $G \times G$ on $G$). Yet more generally, one can define the ring of conditions for a so-called {\it spherical homogeneous space}. Beside matrix groups, other examples include Grassmannians $\textup{Gr}(n, k)$ or the flag variety $\mathcal{F}\ell_n$.} 

{The following shows that the ring of conditions is not well-defined for all groups. For instance, if instead of the multiplicative group $((\C^*)^n, \times)$ we consider the additive group $(\C^n, +)$. This example goes back to De Concinit and Procesi. Take the $3$-dimensional affine space $\C^3$ regarded as an additive group.  
Consider the surface (quadric) $S$ in $\C^3$ defined by the equation $y = zx$. For fixed $z$ the intersection of a horizontal plane $z=a$ and $S$ is the line $y=ax$. This shows that all the lines $y=ax$ must be equivalent in the ring of conditions of $\C^3$.
On the other hand we claim that two skew lines $\ell_1$ and $\ell_2$ cannot be equivalent. This is because one can find a $2$-dimensional plane $P$ such that any translate of $P$ intersects $\ell_1$ but no translate of $P$ intersects $\ell_2$ unless it contains $\ell_2$. The contradiction shows that the ring of conditions is not well-defined for $\C^3$ (see \eqref{equ-numerical-equiv}).}

In the next few sections we consider a piecewise linear analogue of the ring of conditions of $(\C^*)^n$. We denote this piecewise linear analogue by $\mathcal{TR}$.
The elements of the ring $\mathcal{TR}$ are (equivalence classes) of so-called {\it balanced fans}. They are defined in Section \ref{sec-balanced-fan} and their intersection is defined in Section \ref{sec-intersec-balanced-fans}. The notion of tropical variety, introduced in Section \ref{sec-top-var}, makes a connection between the ring $\mathcal{R}((\C^*)^n)$ (defined using algebraic geometry) and the ring $\mathcal{TR}$ (defined using convex and piecewise linear geometry). A main result is that these two rings are in fact isomorphic. {The other main result is that $\mathcal{R}((\C^*)^n)$ is isomorphic to the ring constructed out of the volume polynomial on convex polytopes (Sections \ref{sec-ring-polynomial}) and \ref{sec-polytope-alg-ring-cond}.}

\section{Weighted fans and balancing condition}  \label{sec-balanced-fan}
We begin by recalling the definition of a fan. 
A {\it rational convex polyhedral cone} in $\R^n$ is a convex cone generated by a finite number of vectors in $\Q^n$. It is called strictly convex if it does not contain a line passing through the origin (in other words, it does not contain a $180^\circ$ degree angle). 

\begin{Def}[Fan]  \label{def-fan}
A {\it fan} $\Sigma$ in $\R^n$ is a finite collection of strictly convex rational polyhedral cones in $\R^n$ such that: (1) if $\sigma \in \Sigma$ then any face of $\sigma$ also belongs to $\Sigma$, and (2) if $\sigma_1, \sigma_2 \in \Sigma$ then $\sigma_1 \cap \sigma_2$ is a face of both $\sigma_1$, $\sigma_2$ and belongs to $\Sigma$ as well. The union of all the cones in a fan $\Sigma$ is called the {\it support} of the fan $\Sigma$ and denoted by $|\Sigma|$. Notice that different fans can have the same support. A fan $\Sigma$ is called {\it complete} if $|\Sigma| = \R^n$. 
\end{Def}

Let $\Sigma$ be a fan in $\R^n$. For each $0 \leq i \leq n$ let us denote the set of cones in $\Sigma$ of dimension $i$ by $\Sigma(i)$. A fan $\sigma$ is called a $d$-fan if all the maximal cones in $\Sigma$ have dimension $d$.  

\begin{Def}[Weighted fan]
Let $\Sigma$ be a $d$-fan in $\R^n$.
A {\it weighting} on $\Sigma$ is a function $c: \Sigma(d) \to \R$. We call a fan $\Sigma$ equipped with a weighting $c$ a {\it weighted fan}. If the values of $c$ are integers then we call $c$ an {\it integral weighting} and $(\Sigma, c)$ an {\it integral weighted fan}.
\end{Def}

Now we would like to define when a weighting function $c: \Sigma \to \R$ is {\it balanced}. We start with balancing for $1$-fans. 

\begin{Def}[Balanced $1$-fan]   \label{def-balanced-1-fan}
Let $\Sigma$ be a weighted $1$-fan with weighting function $c$. That is, $\Sigma$ consists of a union of rays $\rho_1, \ldots, \rho_s$ through the origin with corresponding weights $c(\rho_1), \ldots, c(\rho_s)$. For each ray $\rho_i$ let $\xi_i$ be the primitive vector along $\rho_i$, i.e. $\xi_i$ is the vector along $\rho_i$ with integral length $1$, or in other words, $\xi_i$ is the smallest nonzero lattice vector along $\rho_i$. We say that $(\Sigma, c)$ is {\it balanced} if the following vector equation holds:
\begin{equation}   \label{equ-balanced-1-fan}
\sum_{i} c(\rho_i)\, \xi_i = 0.
\end{equation}
\end{Def}

Next we extend the definition of a balanced fan to higher dimensional fans.

\begin{Def}[Balanced $d$-fan]    \label{def-balanced-d-fan}
Let $(\Sigma, c)$ be a weighted $d$-fan with weighting function $c$. We say that $(\Sigma, c)$ is {\it balanced} if the following holds. Let $\tau$ be any codimension $1$ cone in $\Sigma$ (i.e. $\dim(\tau) = d-1$). Let $\sigma_1, \ldots, \sigma_s \in \Sigma(d)$ be $d$-dimensional cones adjacent to $\tau$. Let $N_\tau$ (respectively $N_{\sigma_i}$) be the $(d-1)$-dimensional (respectively $d$-dimensional) lattice generated by $\tau \cap \Z^n$ (respectively $\sigma_i \cap \Z^n$). Then the quotient lattice $N_{\sigma_i}/N_\tau$ has rank $1$, i.e. $N_{\sigma_i}/N_\tau \cong \Z$. Let $\xi_{\sigma_i, \tau}$ be a (non-unique) lattice point in $\sigma_i$ whose image generates the quotient $N_{\sigma_i}/N_\tau$. The balancing condition requires that the vector:
\begin{equation}   \label{equ-balanced-d-fan}
\sum_{i} c(\sigma_i)\, \xi_{\sigma_i, \tau}
\end{equation}
lies in the linear span of $\tau$.
\end{Def}

Finally, we define the notion of normal fan of a convex polytope. Let $\Delta \subset \R^n$ be a full dimensional rational convex polytope. For each  vertex $v$ in $\Delta$ (i.e. a face of dimension $0$), let $C_v$ be the cone at this vertex, namely, $C_v$ is the cone  in $\R^n$ (with apex at the origin) generated by the shifted polytope $(-v) + \Delta$. The {\it dual cone} $\check{C}_v$ is defined as:
$$\check{C}_v = \{ x \in \R^n \mid x \cdot y \geq 0, ~\forall y \in \Delta \}.$$
One shows that the dual cones $\check{C}_v$, for all the vertices $v$ of $\Delta$, fit together to form a complete fan $\Sigma_\Delta$. The fan $\Sigma_\Delta$ is usually called the {\it normal fan} of the polytope $\Delta$. The rays (i.e. the $1$-dimensional cones) in the normal fan are the inward normals to the facets of the polytope $\Delta$.

In the next section (Section \ref{sec-Pascal}) we see that the collection of rays in the normal fan of a lattice polytope has a natural weighting with respect to which it is a balanced $1$-fan. This is related to a classical theorem from geometry known as Pascal's theorem. Similarly, the $(n-1)$-skeleton of the normal fan of a lattice polytope (i.e. the collection of $(n-1)$-dimensional cones in it) can also be equipped with a natural balanced weighting. Each $(n-1)$-dimensional cone in the normal fan is orthogonal to a side (i.e. a $1$-dimensional face) of the polytope. Define the weight of an $(n-1)$-dimensional cone to be the integral length of its corresponding side. One can verify that this gives a balanced weighting on the $(n-1)$-skeleton of the normal fan. This is related to Section \ref{sec-trop-var-Newton-polytope}.

\section{Pascal's theorem}   \label{sec-Pascal}
A classical theorem from geometry (attributed to Pascal) tells us a natural way to construct a balanced $1$-fan from a given lattice polytope (Theorem \ref{th-Pascal-int}).
 
\begin{Th}[Pascal's theorem]    \label{th-pascal}
Let $\Delta$ be a polytope in $\R^n$ with facets $F_1, \ldots, F_s$. For each facet $F_i$ let $n_i$ be the unit normal vector to $F_i$. Also let $\vol_{n-1}(F_i)$ denote the $(n-1)$-dimensional volume of the facet $F_i$. Then the following vector equation holds:
\begin{equation}   \label{equ-Pascal-thm}
\sum_i \vol_{n-1}(F_i) \, n_i = 0.
\end{equation} 
\end{Th}

\begin{proof}[Proof of Pascal's theorem for polygons in $\R^2$]
Rotate each vector $\ell(F_i) n_i$ by $90$ degrees (counter-clockwise). Here $\ell$ denotes the length which is the $1$-dimensional volume. Then the equation \eqref{equ-Pascal-thm} becomes the sum of edge vectors of the polygon $\Delta$ (oriented counter-clockwise) which is clearly equal to $0$. 
\end{proof}

\begin{proof}[Proof of Pascal's theorem in $3$-dimension using physics]
Consider an infinite pool filled with some incompressible fluid. In this pool consider an imaginary region bounded by the polytope $\Delta$. 
According to Pascal's law the external force (caused by the liquid pressure) applied to each facet is equal to the area of the facet times the unit normal vector to the facet. Since the total external force does not move the polytope, the sum of these forces must be equal to $0$.
\end{proof}

\begin{proof}[Sketch of an elementary proof of Pascal's theorem in $3$-dimension]
To prove the vector equality \eqref{equ-Pascal-thm} it is enough to show that for any vector $u \in \R^n$ we have: $$\sum_i \vol_{n-1}(F_i) \, (n_i \cdot u) = 0.$$ For a vector $u \in \R^n$ let $P_u$ be a plane orthogonal to $u$ and let $\pr_u$ denote the orthogonal projection onto $P_u$. Let $\Delta_0 = \pr_u(\Delta)$. Every $x$ in the interior of $\Delta_0$ has $2$ preimages $x'$, $x''$ in the boundary of $\Delta$. Let $x'$ be the one lying over and $x''$ the one lying under. Let $F'_1, \ldots, F'_k$ (respectively $F''_1, \ldots, F''_\ell$) be the facets lying over (respectively lying under). Since the projections of the $F'_i$ and $F''_j$ both cover $\Delta_0$ we see that the sum of areas of projections of the $F'_i$ is equal to the sum of areas of projections of the  $F''_j$. The equality \eqref{equ-Pascal-thm} follows from this.
\end{proof}

\begin{Prob}
Complete the above proof.
\end{Prob}

In fact, the above proof can be extended to arbitrary dimensions.

\begin{proof}[Proof of Pascal's theorem in general using calculus]
The Pascal theorem follows from divergence theorem applied to constant vector fields equal to standard basis vectors.
\end{proof}



We also have an integral version of Pascal's theorem. First we need few definitions. Recall that a vector $v \in \Z^n$ has integral length $\ell$ if there are $\ell + 1$ lattice points on the line segment joining the origin $0$ and $v$. 
Also let $E \subset \R^n$ be an integral affine subspace, i.e. a rational vector subspace $W \subset \R^n$ shifted by some fixed vector $a \in \Z^n$, i.e. $E = a + W$. Let $\dim(E) = d$. By the {\it integral volume} in the affine subspace $E$ we mean the usual $d$-dimensional volume on $E$ but normalized so that a fundamental domain for the lattice $W \cap \Z^n$ has volume $1$. We denote the integral volume on $E$ by $\widehat{\vol}_E$. 

\begin{Prob}
Show that if $E$ is a codimension $1$ integral affine subspace then $\widehat{\vol}_E = \frac{1}{|\xi|} \vol_E$ where $\xi$ is a primitive vector normal to $E$ (i.e. with integral length $1$).
\end{Prob}

\begin{Prob}[Integral version of Pascal's theorem]  \label{th-Pascal-int}
Let $\Delta \subset \R^n$ be a lattice polytope. Let $F_1, \ldots, F_s$ denote the facets of $\Delta$. For each facet $F_i$ let $\xi_i$ denote the primitive inward normal vector to $F_i$. Then:
\begin{equation}   \label{equ-Pascal-thm-int}
\sum_{i=1}^s \widehat{\vol}_i(F_i) \, \xi_i = 0,
\end{equation} 
where $\widehat{\vol}_i$ denotes the integral volume in the affine span of the facet $F_i$.  
\end{Prob}

\section{Intersection of balanced fans}   \label{sec-intersec-balanced-fans}
Given two balanced fans in $\R^n$ one can define their intersection, which is again a balanced fan. This is usually called the {\it (stable) intersection of fans}. 

Let $(\Sigma, c), (\Sigma', c')$ be balanced fans in 
$\R^n$. Let us assume that $\Sigma$, $\Sigma'$ have complementary dimensions, that is, $\Sigma$ is a $d$-fan and $\Sigma'$ is a $d'$-fan and $d+d' = n$. Take a vector $a' \in \R^n$ and consider the translated fan $a' + \Sigma'$. We say that $\Sigma$ and $a'+\Sigma'$ {\it intersect transversely} if for any point $p \in \Sigma \cap (a'+\Sigma')$, there exists top dimensional faces $\sigma \in \Sigma$, $\sigma' \in \Sigma'$ such that $p$ belongs to the relative interior of $\sigma$ and $a'+\sigma'$. One can show that if two fans $\Sigma$, $\Sigma'$ have complementary dimensions, then for almost all $a' \in \R^n$, $\Sigma$ and $a' + \Sigma'$ intersect transversely.

\begin{Ex}[Transverse and non-transverse intersections of fans.]   \label{ex-tranverse-int-fan}
Consider the $1$-fan in Figure \ref{fig-trop-line}. The intersection of this fan with itself shifted by the vector $(1,1)$ is non-transverse, while its intersection with itself shifted by the vector $(1, 2)$ is transverse.
\end{Ex}

Let $\Lambda_\sigma$ and $\Lambda_{\sigma'}$ denote the lattices $\sigma \cap \Z^n$ and $\sigma' \cap \Z^n$ respectively. 

\begin{Def}[Intersection number of balanced fans with complementary dimensions] \label{def-int-balanced-fan}
Let $(\Sigma, c)$ and $(\Sigma', c')$ be balanced fans with complementary dimensions and let $a' \in \R^n$ be such that $\Sigma$ and $a'+\Sigma'$ intersect transversely. With notation as above, for each $p \in \Sigma \cap (a'+\Sigma')$ define the {\it intersection multiplicity} $m_p$ by:
$$m_p = c(\sigma) c'(\sigma') [\Z^n : \Lambda_{\sigma}+\Lambda_{\sigma'}].$$
The {\it intersection number} of $(\Sigma, c)$ and $(\Sigma', c')$ is then defined to be:
$$(\Sigma, c) \cdot (\Sigma', c') = \sum_{p \in \Sigma \cap (a'+\Sigma')} m_p.$$
\end{Def}

\begin{Rem}   \label{rem-int-number-well-defined}
One proves that the intersection number is well-defined, that is, it is independent of the choice of a generic vector $a' \in \R^n$. For this one needs to use the assumption that $(\Sigma, c)$, $(\Sigma', c')$ are balanced. 
\end{Rem}

The intersection of balanced fans can be extended to all balanced fans (not necessarily with complementary dimension).

\begin{Prob}
Show that the balancing condition (Definition \ref{def-balanced-d-fan}) is equivalent to the statement that the intersection number of the fan with a plane of complementary dimension is well-defined.
\end{Prob}

\begin{figure}[ht]
\includegraphics[width=7cm]{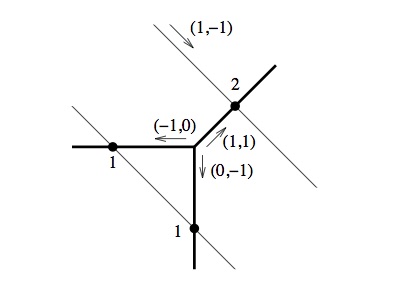} 
\caption{Intersection number of a $1$-fan with a line (in the $2$-dimensional plane) and its invariance under shifting of the line} 
\end{figure} 

Let $\Sigma$ be a fan in $\R^n$. Recall that the {\it support} $|\Sigma|$ is the union of all the cones in $\Sigma$. A {\it subdivision} $\Sigma'$ of $\Sigma$ is a fan obtained by dividing cones in $\Sigma$ into smaller cones. In other words, $\Sigma$ and $\Sigma'$ have the same support and each cone in $\Sigma'$ is contained in some cone in $\Sigma$.

Let $(\Sigma_1, c_1)$, $(\Sigma_2, c_2)$ be two weighted $d$-fans. We say that they are {\it equivalent} if the following hold:
\begin{itemize}
\item[(1)] The two fans $\Sigma_1$ and $\Sigma_2$ have the same support, i.e. $|\Sigma_1| = |\Sigma_2|$.
\item[(2)] The weighting functions $c_1$ and $c_2$ induce the same weighting function on every common subdivision of $\Sigma_1$ and $\Sigma_2$.
\end{itemize}

It is straightforward to verify that the (stable) intersection of fans agrees with the equivalence of balanced fans above.

We can also define addition of equivalence classes of weighted fans. Let $(\Sigma_1, c_1)$, $(\Sigma_2, c_2)$ be two weighted $d$-fans. Let $\Sigma$ be a common subdivision of $\Sigma_1$ and $\Sigma_2$. We then define the sum of $(\Sigma_1, c_1)$, $(\Sigma_2, c_2)$ to be the equivalence class represented by $(\Sigma, c_1 + c_2)$.
 
\begin{Def}[Ring of balanced fans]
For $0 \leq d \leq n$, let $\mathcal{TR}_d$ denote the collection of all balanced $d$-fans in $\R^n$ up to the above equivalence. Let $\mathcal{TR} = \bigoplus_{d=0}^n \mathcal{TR}_d$. 
The set $\mathcal{R}$ together with addition and intersection of balanced fans form a ring which we call the {\it ring of balanced fans}. 
\end{Def}

\section{Tropical variety of a subvariety in $(\C^*)^n$} \label{sec-top-var}
In this section we discuss the notion of tropical variety of a subvariety in $(\C^*)^n$. This is the key idea to translate intersection theoretic data in the ring of conditions $\mathcal{R}((\C^*)^n)$ to piecewise linear data in the ring of balanced fans $\mathcal{TR}$.\footnote{In fact, in our notation $\mathcal{TR}$, the letter $\mathcal{T}$ stands for ``tropical''.}

To a subvariety $X \subset (\C^*)^n$ we can associate its {\it Bergman set}, also called {\it tropical variety of $X$}, which we denote by $\trop(X)$. It is a union of convex polyhedral cones in $\R^n$ and encodes the asymptotic directions on which $X$ can approach ``infinity''.

Let us define $\trop(X)$ more precisely. We say that a lattice vector $k = (k_1, \ldots, k_n) \in \Z^n$ is an {\it asymptotic direction} for $X$, if there is a meromorphic map 
$f = (f_1, \ldots, f_n): (\C, 0) \to X \subset (\C^*)^n$ such that $k$ is the leading exponent of the Laurent series expansion of $f$ at $0$. That is, 
for any $1 \leq i \leq n$ we have $f_i(t) = a_i t^{k_i} + \textup{ higher terms}$.

\begin{Def}[Bergman set or tropical variety of a subvariety]   \label{def-trop-var}
The {\it tropical variety} $\trop(X)$ is defined as the closure of the set of all $ck$ for all $c >0$ and all asymptotic directions $k \in \Z^n$ for $X$.
\end{Def}

The following is a well-known result that goes back to Bergman (\cite{Bergman}).
\begin{Th}    \label{th-trop-var-fan}
Suppose $X \subset (\C^*)^n$ is a subvariety and each irreducible component of $X$ has (complex) dimension $d$. Then $\trop(X)$ is the support of a fan in $\R^n$ and all maximal cones in $\trop(X)$ have (real) dimension $d$. In other words, $\trop(X)$ is a union of finite number of $d$-dimensional convex polyhedral cones in $\R^n$ (that fit together to form a fan).
\end{Th}

\begin{Ex}
Consider the line $x+y+1 = 0$ in $(\C^*)^2$. The tropical variety of this line consists of the union of $3$ rays as in Figure \ref{fig-trop-line}.
\begin{figure}
\includegraphics[width=5cm]{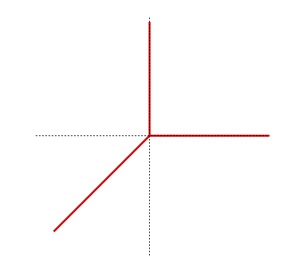}
\caption{Topical variety of a line in the plane. Clearly, it is the support of a $1$-fan.}  \label{fig-trop-line}
\end{figure} 
\end{Ex}

\section{Relation between tropical variety and Newton polytope}   \label{sec-trop-var-Newton-polytope}
Let $f(x_1, \ldots, x_n) = \sum_{\alpha = (a_1, \ldots, a_n) \in \Z^n} c_\alpha x_1^{a_1} \cdots x_n^{a_n}$ be a Laurent polynomial. Let $X_f = \{(x_1, \ldots, x_n) \mid f(x_1, \ldots, x_n) = 0\} \subset (\C^*)^n$ be the hypersurface defined by $f$. Also let $\Delta(f)$ denote the Newton polytope of $f$ (as defined in Definition \ref{def-Newton-polytope}), i.e. the convex hull of the exponents $\alpha$ that appear in $f$, and also let $\trop(X_f)$ be the tropical variety of the hypersurface $X_f$ (Definition \ref{def-trop-var}). There is a direct relationship between these two concepts (they are dual to each other) in the sense we explain below.


\begin{Th}  \label{th-tropical-hyper-Newton-polytope}
Let $f$ be a Laurent polynomial. Let $\Delta = \Delta(f)$ be the Newton polytope of $f$ and $\Sigma = \Sigma_{\Delta_f}$ its corresponding normal fan. With notation as above, the tropical variety $\trop(X_f)$ of the hypersurface $X_f$ is the $(n-1)$-skeleton $\Sigma(n-1)$ of the fan $\Sigma$, that is, the union of all the cones in $\Sigma$ of dimension $n-1$. 
\end{Th}


\section{The BKK theorem revisited}  \label{sec-BKK-revisited}
We can restate the BKK theorem (Section \ref{sec-BKK}) in terms of intersection of balanced fans.
Let $f_1, \ldots, f_n$ be Laurent polynomials in $x=(x_1, \ldots, x_n)$ with Newton polytopes $\Delta_1, \ldots, \Delta_n$ respectively. For each $i$, let $H_i = \{ x \in (\C^*)^n \mid f_i(x) = 0\}$ be the hypersurface defined by $f_i$. 
One can reformulate the BKK theorem in the following way:

\begin{Th}   \label{th-BKK-balanced-fan}
{The number of points of intersection of generic
hypersurfaces $H_i$ is equal to
to the intersection of the balanced fans corresponding to the $H_i$.}
\end{Th}

{Both statements of the BKK theorem equally work in the ring of condition. That is, one can put the hypersurfaces in general position only by moving them using the group elements.}
\begin{Th}  \label{th-BKK-ring-of-conditions}
{The product of the hypersurfaces in the ring of conditions R is
equal to $n!$ times the mixed volume $V(\Delta_1, \ldots, \Delta_n)$. Moreover, this number is equal to the intersection number of the balanced fans corresponding to the $H_i$.}
\end{Th}

\begin{Ex}[Intersection of two general curves in the plane]
Let $C_1$, $C_2$ be two curves in $(\C^*)^2$ defined by polynomials $f$, $g$ of degrees $d$, $e$ respectively. Let us assume that the coefficients of $f$ and $g$ are general. In this case the classical Bezout theorem implies that the number of intersections of $C_1$ and $C_2$ is equal to $de$. The Newton polytopes of $f$ and $g$ are $d\Delta$ and $e\Delta$ where $\Delta$ is the triangle with vertices $(0,0)$, $(0,1)$ and $(1,0)$. One computes that $2!$ times $V(d\Delta, e\Delta) = de$ which agrees with the Bezout theorem. On the other hand, the tropical fans of $C_1$ and $C_2$ are the fan in Figure \ref{fig-trop-line} with weights of rays equal to $d$ and $e$ respectively. One verifies that the tropical intersection number of these two fans is also equal to $de$. 
\end{Ex}


\section{Ring of conditions and balanced fans}
The following theorem gives a combinatorial/convex geometric 
description of the ring of conditions of $(\C^*)^n$ as the ring of balanced fans.

\begin{Th}
The ring of conditions of the torus $(\C^*)^n$ coincides with the algebra of balanced fans in $\R^n$.
\end{Th}

In the remainder of these notes we give another description of the ring of conditions $\mathcal{R}((\C^*)^n)$ in terms of volume of polytopes (see \cite{Kaz-Khov}). The next section (Section \ref{sec-ring-polynomial}) discusses a general construction that associates a ring to a homogeneous polynomial on a vector space. This will be applied in Section \ref{sec-polytope-alg-ring-cond} to the volume polynomial on the vector space of virtual polytopes to give an alternative description of the ring of conditions. 

\section{Algebra associated to a polynomial}   \label{sec-ring-polynomial}

{Let $V$ be a vector space over the field $\R$. Let us recall what it means for a function $P: V \to \R$ to be {\it polynomial}. First let us assume that $V$ is finite dimensional and fix a vector space basis $\mathcal{B} = \{b_1, \ldots, b_m\}$. Then every element $v \in V$ can be 
written uniquely as a linear combination $v = x_1b_1 + \cdots + x_mb_m$, $x_i \in \R$. Then $f$ is a polynomial on $V$ of degree $n$ if the function $(x_1, \ldots , x_m) \mapsto P(x_1b_1+ \cdots +x_mb_m)$ is a polynomial from $\R^m$ to $\R$ of degree $n$. It is easy to check that this concept is independent of the choice of the basis, i.e. if $P$ is a polynomial with respect to a basis $\mathcal{B}$ then it is a polynomial with respect to any other basis. Now, if $V$ is an infinite dimensional vector space we say that a function $P: V \to \R$ is a polynomial function of degree $\leq n$ if its restriction to any finite dimensional subspace is a polynomial function of degree $\leq n$. We also recall that $P$ is called a {\it homogeneous polynomial of degree $n$} if for any $v \in V$ and any $c \in \R$ we have $P(cv) = c^n P(v)$. Equivalently, we can say that $P$ is a homogeneous polynomial on $V$ of degree $n$ if there is an $n$-linear map $F: \underbrace{V \times \cdots \times V}_{n} \to \R$ such that $$P(v) = F(v, \ldots, v),$$ for all $v \in V$.}

\begin{Prob} \label{prob-ext-poly}
Let $C \subset V$ be a convex cone (with apex at the origin) that spans $V$ (i.e. $C$ is full dimensional). Let $P: C \to \R$ be a function on $C$ such that there exists a polynomial function $Q: V \to \R$ that agrees with $P$ restricted to the cone $C$. Show that this polynomial $Q$ is unique, in other words, show that $P$ has a unique extension to a polynomial function on the whole $V$.
\end{Prob}

Now let us consider the algebra $\D = \D_V$ of constant coefficient differential operators on the vector space $V$.
For a vector $v \in V$, let $L_v$ be the 
differentiation operator (Lie derivative) on the space of polynomial functions on $V$ defined as follows. Let $f$ be a polynomial function on $V$. Then: $$L_v(f)(x) = \lim_{t \to 0} \frac{f(x+tv) - f(x)}{t}.$$
The algebra $\D$ is defined to be the commutative algebra generated by multiplications by scalars and by the Lie derivatives $L_v$ for all $v \in V$.

 When $V \cong \R^n$ is finite dimensional, $\D$ can be realized as follows: Fix a basis for $V$ and let $(x_1, \ldots, x_n)$ denote the coordinate functions with respect to this basis. Each element of $\D$ is then a polynomial, with constant coefficients, in the differential operators $\partial/\partial x_1, \ldots, \partial/\partial x_n$. That is: $$\D = \{ f(\partial/\partial x_1, \ldots, \partial/\partial x_n) \mid f = \sum_{\alpha = (a_1, \ldots, a_n)} c_\alpha x_1^{a_1} \cdots x_n^{a_n} \in \R[x_1, \ldots, x_n]\}.$$
{Thus, as an algebra $D$ is {\it very similar} to the algebra of polynomials on $V$ (in fact it is {\it dual} to this algebra).}

As before, let $P: V \to \R$ be a homogeneous polynomial 
function (of degree $n$) on a vector space $V$. To $(V, P)$ we associate an algebra $A_P$ as follows. As above, let $\D$ be the algebra of constant coefficient differential operators on the vector space 
$V$. Also let $I$ be the set of all differential operators $D \in \D$ such that $D \cdot P = 0$, i.e. those differential operators that annihilate $P$.

\begin{Prob}
Show that $I$ is an ideal in the algebra $\D$.
\end{Prob}

\begin{Def}[Algebra associated to a homogeneous polynomial]
We call the quotient algebra $A_P = \D / I$, {\it the algebra associated to the polynomial $P$}. 
\end{Def}

\begin{Prob}
Show that the algebra $A = A_P$ has a natural grading by $\Z_{\geq 0}$ (for this one shows that $I$ is a homogenous ideal). Let $A_i$ denote the $i$-th graded piece of $A$. Show that $A_{0} \cong A_{n} \cong \R$ and $A_{i} = \{0\}$, for any $i>n$.
Finally, show that the algebra $A$ has {\it Poincare duality}, i.e. for any $0 \leq i \leq n$, the bilinear map $A_{i} \times A_{n-i} \to A_{n} \cong \R$ given by multiplication, is non-degenerate. Thus, we have $\dim_\R(A_{i}) = \dim_\R(A_{n-i})$.
\end{Prob}

\section{Polytope algebra and ring of conditions}  \label{sec-polytope-alg-ring-cond}
Recall from Section \ref{sec-mixed-vol} that $\mathcal{P}$ denotes the space of all convex polytopes in $\R^n$.  
We first introduce the vector space of virtual polytopes. We need the following.
\begin{Prob}   \label{prob-Minkoski-sum-cancelative}
Show that the Minkowski sum is cancelative, i.e. if $\Delta_1, \Delta_2, \Delta \in \mathcal{P}$ are convex polytopes with $\Delta_1 + \Delta = \Delta_2 + \Delta$ then $\Delta_1 = \Delta_2$. 
\end{Prob}

The above problem shows that the set $\mathcal{P}$, of convex polytopes in $\R^n$, together with the Minkoski sum and multiplication with positive scalars can be formally extended to an (infinite dimensional) vector space $\mathcal{V}$. This vector space $\mathcal{V}$ is called the vector space of {\it virtual polytopes}. The elements in this vector space, namely {\it virtual polytopes}, are formal differences $\Delta_1 - \Delta_2$ of polytopes $\Delta_1, \Delta_2 \in \mathcal{P}$. Two formal differences $\Delta_1 - \Delta_2$ and $\Delta_1' - \Delta_2'$ are considered equal if $\Delta_1 + \Delta_2' = \Delta_1' + \Delta_2$. The volume and mixed volume functions extend to the vector space $\mathcal{V}$ of virtual polytopes in the obvious way (see Problem \ref{prob-ext-poly}).

For our application of convex polytopes to the ring of conditions we only need linear combinations of lattice polytopes in $\R^n$ (this is essentially because Newton polytopes of Laurent polynomials are all lattice polytopes by definition). We call the subspace of $\mathcal{V}$ spanned by lattice polytopes by $\mathcal{L}$. 

As a special case of the construction in Section \ref{sec-ring-polynomial}, to the vector space $\mathcal{L}$ and the volume polynomial $\vol$ we can associate an algebra $A_{\vol}$. We call it the {\it polytope algebra}. It turns out that, similar to the ring of balanced fans, the multiplication in the polytope algebra contains a great deal of information about number of solutions of systems of Laurent polynomial equations and more generally intersection numbers of subvarieties in $(\C^*)^n$. More precisely, the polytope algebra gives another description of the ring of conditions of $(\C^*)^n$ (see \cite{Kaz-Khov}):

\begin{Th}
The ring of conditions of the torus $(\C^*)^n$ is isomorphic to the polytope algebra $A_\vol$.
\end{Th}

\end{document}